\documentclass[11pt]{article}     
\usepackage{graphicx}
\usepackage{latexsym}
\usepackage{amsmath}

\usepackage{authblk}

\usepackage{amsthm}
\usepackage{dsfont}

\newcommand{\R}{{\ensuremath{\mathds{R}}}}

\usepackage{enumerate}

\usepackage{bbm}
\usepackage{graphics}
\usepackage{setspace}

\newtheorem{theorem}{Theorem}[section]

\newtheorem{lemma}[theorem]{Lemma}

\theoremstyle{definition}
\newtheorem{definition}[theorem]{Definition}

\numberwithin{equation}{section}

\newcommand\Tr{\operatorname{Tr}}

 \newcommand{\eps}{\varepsilon}

 \newcommand{\Prob}{\mathbf{P}}

 \newcommand{\calL}{\mathcal{L}}

 \newcommand{\F}{\mathcal{F}}

 \newcommand{\U}{\mathcal{U}}
 \newcommand{\M}{\mathcal{M}}

\newcommand{\Fil}{\mathds{F}}
\newcommand{\Exp}{\mathds{E}}

\def\calX{{\mathcal X}}

\def\rmX{{\mathrm X}}

\def\calY{{\mathcal Y}}

\def\calZ{{\mathcal Z}}

\def\calW{{\mathcal W}}

\def\B{{\mathcal B}}

\def\calM{{\mathcal M}}

\def\calC{{\mathcal C}}

\newcommand{\abs}[1]{\left\vert#1\right\vert}
\newcommand{\set}[1]{\left\{#1\right\}}
\newcommand{\seq}[1]{\left<#1\right>}

\title{On the LP formulation in measure spaces of optimal control problems for jump-diffusions}

\author{\textbf{Rafael Serrano} \thanks{E-mail address: rafael.serrano@urosario.edu.co}}
\affil{\textsc{Universidad del Rosario}\\Calle 12c No. 4-69\\
Bogot\'a, Colombia}

\date{\today}

\begin{document}
\maketitle











\begin{abstract}
In this short note we formulate a infinite-horizon stochastic optimal control problem for jump-diffusions of Ito-Levy type as a LP problem in a measure space, and prove that the optimal value functions of both problems coincide. The main tools are the dual formulation of the LP primal problem, which is strongly connected to the notion of sub-solution of the partial integro-differential equation of Hamilton-Jacobi-Bellman type associated with the optimal control problem, and the Krylov regularization method for viscosity solutions.

\end{abstract}




\section{Introduction}

This short note revisits the infinite-dimensional linear programming (LP) approach to stochastic optimal control problems. We reformulate the problem of minimizing a infinite-horizon cost functional for controlled jump-diffusions of Ito-Levy type over a set of admissible controls as a linear program in a certain measure space. The linear objective function is the integral of the cost function against the occupation measure of the controlled process. The main challenge in this
LP approach is to prove equality of the optimal value functions of the
original control problem $V(x)$ and the associated infinite-dimensional linear program $\rho(x),$ and absence of duality gap between the primal
and dual programs (strong duality).

Using measure-valued (relaxed) controls, Stockbridge \cite{stock1990} proved the equality $\rho=V$ for ergodic optimal control of Markov processes and existence of optimal controls. Bhat and Borkar \cite{bhattborkar} and Kurtz and Stockbridge \cite{kurtzstock} extended these results to the case of feedback controls for time-inhomogeneous finite horizon and discounted infinite horizon problems. Cho and Stockbridge \cite{cho}, Taksar \cite{taksar} and Helmes and Stockbridge \cite{helmstock} obtained similar results for optimal stopping and singular control problems.

More recently, using  the dual formulation of the primal LP problem and viscosity solution theory, Buckdahn et al \cite{bgq2011} proved the equality $\rho=V$ in the case of optimal control diffusions with compact state space. Goreac and Serea \cite{gs2011} proved the same result for finite-horizon and optimal stopping problems. In this paper we show that this approach can be easily extended to the jump-diffusion case. We emphasize that our proof does not present any significant innovation as we follow closely the arguments in the proof of Theorem 6.4 in Jakobsen et al \cite{jklc2008}. However, to the best of our knowledge, this is the first paper that deals with the LP approach to stochastic optimal control problems for jump-diffusions.

Let us briefly describe the contents of this paper. In Section 2 we introduce the setting for the optimal control problem of jump-diffusions of It\^{o}-Levy type and formulate the primal LP problem associated with the optimal control problem and its dual. In Section 3 we recall the definition of viscosity solution for partial integro-differential equations and prove the main result using the Krylov regularization and results from Jakobsen et al \cite{jklc2008}.

\section{Optimal control problem and LP formulation}
Let $(\Omega,\F,\Prob)$ be probability space endowed with a filtration $\Fil=\set{\F_t}_{t\geq 0}$ satisfying the usual conditions, and let $\set{W_t}_{t\geq 0}$ be a standard $d-$dimensional Brownian motion with respect to $\mathds{F}.$

Let $E=\R^N\setminus \set{0}$ and let $\nu(dz)$ be a Levy measure on $\mathcal{B}(E),$ that is, a non-negative $\sigma$-measure satisfying
\[
\int_{E}(\abs{z}^2\wedge 1)\,\nu(dz)<+\infty.
\]
Let $N(dz,dt)$ be a homogeneous Poisson random measure with compensator intensity measure $\nu(dz)\,dt,$ and let $\tilde{N}(dz,dt)$ denote the compensated Poisson random measure $
\tilde{N}(dz,dt):=N(dz,dt)-\nu(dz)\,dt.$

Let $U$ be a compact metric space. For each $\Fil$-adapted $U$-valued control process $u=\set{u_t}_{t\geq 0}$ consider the
controlled Levy-It\^{o} equation
\begin{equation}\label{csIto-Levy}
\begin{split}
dX_t&=b(X_t,u_t)\,dt+\sigma(X_t,u_t)\,dW_t+\int_{E}\eta(X_{t-},u_{t-},z)\,\tilde{N}(dz,dt)\\
X_0&=x.
\end{split}
\end{equation}
The coefficients $b:\R^N\times U\to\R^N,$ $\sigma:\R^N\times U\to\R^{N\times d}$ and $\eta:\R^N\times U\times E\to\R^N$ satisfy conditions (\ref{A2}) and (\ref{A3}) below.
The class $\U(x)$ of admissible control policies is defined as the set of control processes $u=\set{u_t}_{t\geq 0}$ for which equation (\ref{csIto-Levy}) has an unique strong solution $X^{x,u}=\set{X_t^{x,u}}_{t\geq 0}.$

Let $c>0$ be a fixed discount rate and $h:\R^N\times U\to(-\infty,+\infty]$ denote the cost-to-go function. Let $\mathcal{J}$ be the infinite-horizon discounted cost functional
\begin{equation*}
  \mathcal{J}(x,u):=\Exp\left[\int_0^\infty e^{-ct}\,h(X_t^{x,u},u_t)\,dt\right].
\end{equation*}
We will use the following norms
\[
\abs{\phi}_0:=\sup_{x\in\R^N}\abs{\phi(x)}, \ \ [\phi]_1:=\abs{D\phi}_0 \ \ \mbox{ and } \ \ \abs{\phi}_1:=\abs{\phi}_1+[\phi]_1
\]
and assume the following conditions
\begin{enumerate}
  \item The Levy measure $\nu(dz)$ satisfies
\begin{equation}\label{A1}
\int_{\abs{z}\geq 1}e^{m\abs{z}}\nu(dz)<\infty \tag{A1}
\end{equation}
for some $m>0.$

\item There exists $K>0$ such that for all $u\in U$
\begin{equation}\label{A2}
\abs{b(\cdot,u)}_1+\abs{\sigma(\cdot,u)}_1+\abs{c(\cdot,u)}_1+\abs{h(\cdot,u)}_1\le K \tag{A2}
\end{equation}
and
\begin{equation}\label{A3}
\abs{\eta(\cdot,u,z)}_1\le K\left[\abs{z}\mathbf{1}_{\set{0<\abs{z}<1}}(z)+e^{m\abs{z}}\mathbf{1}_{\set{\abs{z}\geq 1}}(z)\right] \tag{A3}
\end{equation}
\end{enumerate}
Condition (\ref{A1}) is equivalent to the Levy process with Levy measure $\nu(dz)$ having finite moments of all orders, see e.g. Applebaum \cite[Section 2.5]{applebaum}.
It is satisfied, for instance, by one-dimensional tempered $\alpha$-stable processes with Levy measure
\[
\nu(dz)=\frac{C_1e^{-\lambda_1z}}{z^{1+\alpha_1}}\mathbf{1}_{\R_+}(z)\,dz
+\frac{C_2e^{-\lambda_2\abs{z}}}{\abs{z}^{1+\alpha_2}}\mathbf{1}_{\R_-}(z)\,dz
\]
with $C_1,C_2\geq 0,$ $\lambda_1,\lambda_2>0$ and $\alpha_1,\alpha_2<2.$ Under conditions (\ref{A1})-(\ref{A3}), for each $u\in\U(x)$ there exists an unique strong solution to equation (\ref{csIto-Levy}) and satisfies the following estimate, see e.g. Applebaum \cite[Section 6.6]{applebaum}
\begin{equation}\label{estX}
  \Exp\left[\sup_{t\in [0,T]}\abs{X_t^{x,u}}^{p}\right]\le C(1+\abs{x}^p)
\end{equation}
for all $p\geq 2.$ The main object of study of this paper is the stochastic optimal control problem
\begin{equation}\label{cp}
  V(x):=\inf_{u\in\U(x)}\mathcal{J}(x,u), \ \ x\in\R^N
\end{equation}
and the following linear programming (LP) formulation: for each $x\in\R^N$ and $u\in\U(x),$ denote with $\gamma^{x,u}$ the expected discounted \emph{occupation measure} on $\mathcal{B}(\R^N\times U)$ defined as
\[
\gamma^{x,u}(Q):=\mathds{E}\left[\int_0^\infty e^{-ct}1_Q(X^{x,u}_t,u_t)\,dt\right], \ \ \ Q\in\mathcal{B}(\R^N\times U).
\]
Using approximation of $h$ by simple functions, it is easy to prove that the occupation measure $\gamma^{x,u}$ satisfies
\[
\mathcal{J}(x,u)=\int_{\R^N\times U} h(y,u)\,\gamma^{x,u}(dy,du).
\]
Let $\mathcal{C}^2_{\rm pol}(\R^N)$ denote the class of $\mathcal{C}^2$-functions $f:\R^N\to\R$ with polynomial growth. For each $u\in U$ fixed, let $A^u+J^u$ denote the partial integro-differential operator
\begin{align*}
A^u f(x)&:=\seq{b(x,u),D f(x)}+\frac{1}{2}\Tr[\sigma(x,u)\sigma(x,u)^*D^2 f(x)],\\
J^u f(x)&:=\int_E\left\{f(x+\eta(x,u,z))-f(x)-\mathbf{1}_{\set{\abs{z}<1}}\seq{\eta(x,u,z),Df(x)}\right\}\,\nu(dz).
\end{align*}
for $f\in\mathcal{C}^2_{\rm pol}(\R^N).$ Here $Df(x)$ and $D^2f(x)$ denote the vector and square matrix of first and second-order partial derivatives of $f$ respectively.

Notice that the integral term in the operator $J^u$ is well-defined due to the exponential decay of the Levy measure $\nu(dz)$ at infinity (see Assumption A.1) and the fact that the singularity at $z=0$ is integrable for any $f\in\mathcal{C}^2(\R^N),$ see e.g. Applebaum \cite[Section 3.3]{applebaum}. 


Using It\^{o}'s formula for Levy-It\^{o} processes, Kunita's inequality and estimate (\ref{estX}), for any $f\in\mathcal{C}^2_{\rm pol}(\R^N)$ and $T>0,$ we have
\[
\Exp[e^{-cT}f(X^{x,u}_T)]-f(x)=\Exp\left[\int_0^T e^{-ct}\left\{[(A+J)f](X^{x,u}_t,u_t)-cf(X^{x,u}_t)\right\}\,dt\right].
\]
Also from estimate (\ref{estX}) we have
\[
\lim_{T\to\infty}\Exp[e^{-cT}f(X^{x,u}_T)]=0.
\]
By the dominated convergence theorem, taking the limit as $T\to\infty$ it follows
\begin{equation*}
\Exp\left[\int_0^\infty e^{-ct}[cf-(A+J)f](X^{x,u}_t,u_t)\,dt\right]=f(x).
\end{equation*}
We have proved that the occupation measure $\gamma^{x,u}$ satisfies the linear constraint,
\begin{equation}\label{eqlin1}
\int_{\R^N\times U}[cf-(A^u+J^u)f](y)\,\gamma^{x,u}(dy,du)=f(x), \ \ \  \forall f\in\mathcal{C}^2_{\rm pol}(\R^N).
\end{equation}
This suggests to consider the following LP problem over the vector space $\calM_{\rm b}(\R^N\times U)$ of finite signed measures on $\B(\R^N\times U)$
\begin{equation}\label{lp0}
\begin{split}
   \rho(x):=&\inf\ \int_{\R^N\times U}h(y,u)\,\mu(dy,du)\\
        &\text{subject to } \ \mu\in \calM_{\rm b}(\R^N\times U), \ \mu\geq 0\\
        & \mbox{and} \  \int_{\R^N\times U}[cf-(A^u+J^u)f](y)\,\mu(dy,du)=f(x), \ \ \forall f\in \mathcal{C}^2_{\rm pol}(\R^N).
\end{split}
\end{equation}
Clearly, we have $\rho\le V.$ The main purpose of this  note is to prove that in fact equality $\rho=V$ holds.

In order to formulate a LP problem with the linear constraint (\ref{eqlin1}), and its dual, we recall briefly some facts and notation concerning
infinite-dimensional linear programming. Two topological real vector
spaces $\calX,\calY$ are said to form a \emph{dual pair} if
there exists a bilinear form $\langle\cdot,\cdot\rangle:\calX\times\calY\to\mathds{R}$ such
that the mappings $\calX\ni x\to\seq{x,y}\in\R$ for $y\in\calY$
separate points of $\calX$ and the mappings $\calY\ni
y\to\seq{x,y}\in\R$ for $x\in\calX$ separate points of $\calY.$



We endow $\calX$ with the weak topology $\sigma(\calX,\calY),$ i.e. the
coarsest topology for which the maps $\calX\ni x\to\seq{x,y}\in\R$
are continuous for all $y\in\calY.$ For any vector subspace $F\subset\calX,$ let $F^*$ denote the vector space of linear functionals on $F$ with respect to the inherited weak topology. Notice that $\calX^*=\calY$ according to this notation.

Let $\calX^+$ be a \emph{positive cone} in $\calX,$ that is, a  non-empty convex cone with vertex $0.$ This induces the vector order
\[
x_1\geq x_2 \ \mbox{iff} \ x_1-x_2\in\calX^+.
\]
The positive cone $\calY^+$ in $\calY$ is defined as the negative polar
\[
\calY^+:=\{y\in\calY:\seq{x,y}\geq 0 \text{ for all }
x\in\calX^+\}.
\]
Notice that $\calY^+$ is closed, and the identity
\[
\calX^+=\{x\in\calX:\seq{x,y}\geq 0 \text{ for all }
y\in\calY^+\}
\]
holds only if $\calX^+$ is closed.

Let $(\calZ,\calW)$ be another dual pair  and let
$\mathcal{L}:X\to Z$ be a continuous linear operator. We define the adjoint map $\mathcal{L}^*:\calW\to\calY$
via the relation
\[
\seq{x,\calL^*w}:=\seq{\calL x,w},  \ \ \ x\in\calX, \ w\in\calW.
\]
Let $b\in \calZ, c\in \calY$ be given, and consider the \emph{primal} linear problem
\begin{align*}
    (\pi) \ & \text{ minimize} \ \seq{x,c}\\
        & \text{ subject to } \calL x=b, \ x\geq 0.
\end{align*}
The (algebraic) \emph{dual} problem of (P) is given by
\begin{align*}
    (\pi^*) \ & \text{ maximize} \ \seq{b,w}\\
        & \text{ subject to } -\calL^*w+c\geq 0, \ w\in \calW.
\end{align*}
It can be proved that the inequality $\sup\mathrm{(\pi^*)}\le \inf\mathrm{(\pi)}$  holds, see e.g. Anderson and Nash \cite{andnash}. If this inequality, usually referred to as \emph{weak duality}, is strict, then we say that there is a \emph{duality gap} between the primal linear program and its dual.

We now follow the above definitions (see also \cite{hernetal,taksar}) to formulate the dual problem to (\ref{lp0}). We set $\calX=\calM_{\rm b}(\R^N\times U)$ and $\calY=\calC_{\rm b}(\R^N\times U),$ the vector space of bounded continuous functions on $\R^N\times U.$ These spaces form a dual pair with the bilinear form
\[
\seq{f,\mu}:=\int_{\R^N\times U} f(y,u)\,\mu(dy,du), \ \ \ f\in\mathcal{C}_{\rm b}(\R^N\times U), \ \ \mu\in\M_{\rm b}(\R^N\times U).
\]
Finally, let $A+J:D(A+J)=\mathcal{C}^2_{\rm pol}(\R^N)\to \B(\R^N\times U)$ denote the linear operator defined by
\[
[(A+J)f](x,u):=(A^uf+J^uf)(x), \ \ (x,u)\in\R^N\times U.
\]
We take $\calW:=D(A+J)=\calC^2_{\rm pol}(\R^N)$ and $\calZ:=D(A+J)^*.$ We define the operator $\calL:\calX\to\calZ$ via its adjoint operator $\calL^*:\calW\to\calY$ as follows
\[
\calL^*f:=cf-(A+J)f, \ \ \ f\in \calW.
\]
With this notation, the linear program (\ref{lp0}) associated with the control problem (\ref{cp}) now reads
\begin{align*}
   \rho(x)=&\inf\ \seq{h,\mu}\\
        &\text{subject to } \ \mu\in \calM_{\rm b}^+(\R^N\times U)\\
        & \phantom{subject to } \ \calL\mu=\delta_x.
\end{align*}
The dual linear program is given by
\begin{align*}
   \rho^*(x)&=\sup\ \seq{f,\delta_x}\\
        &\ \ \ \text{ subject to } \ f\in \calC^2_{\rm pol}(\R^N)\\
        &\phantom{ subject to } - \mathcal{L}^*f+h\geq 0 \ \
\end{align*}
and the weak duality $\rho^*(x)\le \rho(x)$ holds, for all $x\in\R^N.$ Notice that the restriction $-\mathcal{L}^*f+h\geq 0$ on $f\in\calC^2_{\rm pol}(\R^N)$ is equivalent to
\[
cf(y)-h(y,u)-(A+J)(y,u)\leq 0, \ \ \forall u\in U, \ \ \ y\in\R^N
\]
that is, $f$ is a (smooth) sub-solution with polynomial growth of the elliptic partial integro-differential equation (PIDE) of \emph{Hamilton-Jacobi-Bellman} type
\begin{equation}\label{hjb}
\sup_{u\in U}\set{cV(y)-h(y,u)-[(A+J)V](y)}=0, \ \ x\in\R^N \tag{HJB}.
\end{equation}
Recall that equation (\ref{hjb}) is the equation associated with the dynamic programming approach to existence of optimal feedback controls for problem (\ref{cp}). Using this equation, we can rewrite the dual problem as
\[
\rho^*(x)=\sup\left\{f(x):f\in D(A+J) \text{ is sub-solution of equation (\ref{hjb})}\right\}.
\]
The following is the main result of this paper
\begin{theorem}\label{main}
The value function of the control problem (\ref{cp}) satisfies $V(x)\leq\rho^*(x)$ for all $x\in\R^N.$
\end{theorem}
As an immediate consequence, we have $V=\rho$ (optimal values of the LP problem (\ref{lp0}) and the control problem (\ref{cp}) coincide) and
$\rho=\rho^*$ (absence of duality gap).

\section{Viscosity solutions and Krylov regularization}
Notice that equation (\ref{hjb}) can be rewritten in the form
\begin{equation}\label{PIDE}
  H(x,V(x),DV(x),D^2V(x),V(\cdot))=0, \ \ x\in\R^N \tag{P}
\end{equation}
where for any $(x,r,p,\mathrm{X})\in \R^N\times\R\times\R^N\times\mathrm{S}_N$ and sufficiently regular function $V,$ the nonlinear \emph{Hamiltonian} $H$ is defined by
\begin{align*}
&H(x,r,p,\mathrm{X},V(\cdot))\\ &\phantom{AA}=\sup_{u\in U}\Bigl\{cr-h(x,u)-\seq{b(x,u),p}-\frac{1}{2}\Tr[\sigma(x,u)\sigma(x,u)^*\rmX]-J^uV(x)\Bigr\}.
\end{align*}


Because of the degeneracy in the second order differential operator and the non-local integral operator $J^u$, elliptic equation (\ref{PIDE}) is expected to have only solutions in the viscosity sense. The notion of viscosity solution was introduced by P.L. Lions and M.G. Crandall in the 1980s for first-order Hamilton-Jacobi equations. This was later generalized to second-order equations and more recently to equations with integro-differential operators. The viscosity solution approach provides an appropriate framework to formulate HJB equations for functions that are only assumed to be locally bounded, see e.g. Alvarez \cite{alvarez96},  Bardi \cite{bardi}, Barles et al \cite{bbp97}, Barles and Imbert \cite{bi2008}, Crandall et al \cite{usersguide}, Fleming and Soner \cite{fs1}, and the references therein. 


\begin{definition}

\begin{enumerate}
\item[(i)] A locally bounded upper semi-continuous (resp. lower semi-continuous) function $V:\R^N\to\R$ is a \emph{viscosity sub-solution} (resp. \emph{super-solution}) of $\eqref{PIDE}$ iff for every $x\in\R^N$ and $\phi\in\mathcal{C}_{\rm pol}^2(\R^N)$ such that
\[
V-\phi \ \text{attains a global maximum (resp. minimum) at }x
\]
then
\begin{equation*}
H(x,V(x),D\phi(x),D^2\phi(x),\phi(\cdot))\le 0 \ \ (\text{resp. } \geq 0)
\end{equation*}
\item[(ii)] $V$ is a \emph{viscosity solution} of (P) if it is both sub- and super-solution.
\end{enumerate}
\end{definition}


We follow the regularization method of N.V. Krylov \cite{krylov2000}, originally introduced to estimate the convergence rate of finite-difference approximations of equation (\ref{PIDE}) in the diffusion case. It has been systematically improved by Barles and Jakobsen \cite{bj2002}, and Jakobsen et al \cite{jklc2008}, among others, to obtain monotone approximation schemes for elliptic and parabolic equations of Bellman-Isaacs type.

For $\varepsilon\in(0,1),$ consider the auxiliary perturbed equation
\begin{equation}\label{HJBeps}
\sup_{\abs{e}\leq\varepsilon}H(x+e,V^\varepsilon(x),DV^\varepsilon(x),D^2V^\varepsilon(x),V^\varepsilon(\cdot))=0, \ \ x\in\R^N. \tag{P$_\varepsilon$}
\end{equation}
Notice that this is the HJB equation associated with the control problem with control set $U\times \overline{B(0,\eps)},$ cost function $h(x,u,e):=h(x+e,u)$ and coefficients $b,\sigma$ and $\eta$ defined analogously. Let $V^\eps$ denote the optimal value function
\[
V^\eps(x)=\inf_{(u,e)\in\U^\eps(x)}\Exp\left[\int_0^\infty e^{-ct}h(X_t^{x,u,e}+e_t,u_t)\right], \
 \  \ x\in\R^N,
\]
where $\U^\eps(x)$ is the set of admissible control policies and, for each $(u,e)\in\U^\eps(x),$ $X^{x,u,e}$ denotes the corresponding controlled process. Notice that $\U(x)$ can be seen as subset of $\U^\eps(x)$ by identifying $u\in\U(x)$ with $(u,0)\in\U^\eps(x).$

\begin{lemma}\label{eps}
$V$ (resp. $V^\eps$) is the unique viscosity solution to equation (\ref{PIDE}) (resp. (\ref{HJBeps})) in $\calC_{\rm b}^{0,1}(\R^N)$ and,
for each $x\in\R^N,$ we have $V^\eps(x)\to V(x)$ as $\eps\to 0.$
\end{lemma}

\begin{proof}
Viscosity solution property of the value functions $V$ and $V^\eps$ is a standard well-known result that follows from the dynamic programming principle, see e.g. Pham \cite[Section 3]{pham95}, {\O}ksendal and Sulem \cite[Section 10.3]{oksulem} or Bouchard and Touzi \cite[Section 5.2]{boutouzi}. The second assertion follows from the continuous dependence result in Theorem 6.2 in Jakobsen et al \cite{jklc2008}.
\end{proof}


We also need the following result. For the proof see Lemma 6.3 in Jakobsen et al \cite{jklc2008}.
\begin{lemma}\label{convex}
Convex combination of viscosity sub-solutions of (P) is also a viscosity sub-solution of (P).
\end{lemma}

\begin{proof}[Proof of Theorem \ref{main}]
The proof is largely based on the proof of Theorem 6.4 in \cite{jklc2008}. Using the change of variable $y = x + e,$ the function $V^\varepsilon(\cdot-e)$ satisfies
\[
H(y,V^\varepsilon(y-e),DV^\varepsilon(y-e),D^2V^\varepsilon(y-e),V(\cdot-e))\le 0
\]
in the viscosity sense, for every $\abs{e}\le\varepsilon.$ That is, $V^\varepsilon(\cdot-e)$ is viscosity sub-solution of (\ref{PIDE}) for every $\abs{e}\leq\varepsilon.$ Now, for $\kappa,\eps\in(0,1),$ let $V_{\eps,\kappa}:=V^\varepsilon\ast\phi_\kappa$ where
\[
\phi_\kappa (x):=\frac{1}{\varepsilon^N}\phi\left(\frac{x}{\kappa}\right)
\]
and $\phi$ is a smooth positive function with unit mass and support in $B(0,1).$ For $h<\kappa$ and $\alpha\in h\mathds{Z}^N$ define
\[
Q_h(\alpha):=\alpha+\left[-\frac{h}{2},\frac{h}{2}\right]^N, \ \ \ \ \phi_{\kappa,h}(\alpha):=\int_{Q_h(\alpha)}\phi_{\kappa}(y)\,dy
\]
and
\[
V_{\eps,\kappa,h}(x):=\sum_{\alpha\in h\mathds{Z}^N}V^\varepsilon(x-\alpha)\phi_{\kappa,h}(\alpha).
\]
By properties of mollifiers, $V_{\varepsilon,\kappa,h}\to V_{\varepsilon,\kappa}$ as $h\to 0.$ Moreover, $\phi_{\kappa,h}(\alpha)>0$ only for finitely many $\alpha\in h\mathds{Z}^N$ since $\phi_\kappa$ has compact support, and
\[
\sum_{\alpha\in h\mathds{Z}^N}\phi_{\kappa,h}(\alpha)=\int_{\R^N}\phi_{\kappa}(y)\,dy=1.
\]
By Lemma \ref{convex}, $V_{\varepsilon,\kappa,h}$ is a viscosity subsolution of (P), and so is its pointwise limit $V_{\varepsilon,\kappa}=\lim_{h\to 0}V_{\varepsilon,\kappa,h}$ by stability results for viscosity subsolutions, see e.g. Barles and Imbert \cite[Section 3]{bi2008}. By the definition of the dual problem, we have $V_{\varepsilon,\kappa}\le\rho^*(x)$ for all $x\in\R^N.$ The desired result follows by taking the limit as $\varepsilon,\kappa\to 0$ in conjunction with Lemma \ref{eps}.
\end{proof}







\end{document}